\documentclass[12pt]{amsart}

\usepackage{latexsym, amsmath, extarrows}

{\theoremstyle{plain}
    \newtheorem{Thm}{\bf Theorem}[section]
    \newtheorem{Prop}[Thm]{\bf Proposition}

    \newtheorem{Cor}[Thm]{\bf Corollary}

}
{\theoremstyle{remark}
    \newtheorem{Rem}[Thm]{\bf Remark}
    \newtheorem{Exa}[Thm]{\bf Example}
}
{\theoremstyle{definition}
    \newtheorem{Def}[Thm]{\bf Definition}
}
\numberwithin{equation}{section}

\newcommand{\length}{\operatorname{\lambda}}

\newcommand{\cH}{{\mathcal H}}

\newcommand{\B}{\mathcal{B}}
\newcommand{\Z}{\mathbb{Z}}
\newcommand{\N}{\mathbb{N}}
\newcommand{\R}{\mathbb{R}}

\newcommand{\excise}[1]{}

\title{Intersection algebras for principal monomial ideals in polynomial rings}
\author[F.~Enescu, S.~Malec]{Florian Enescu and Sara Malec}

\address{Department of Mathematics and Statistics, Georgia State University, Atlanta, 30303}
\email{fenescu@gsu.edu}
\address{Department of Mathematics, University of the Pacific, 3601 Pacific Avenue, Stockton, CA 95211}
\email{smalec@pacific.edu}
\thanks{The first author was partially supported by the NSA Young Investigator Grant H98230-12-1-0206.}
\begin{document}
\bibliographystyle{hplain}
\begin{abstract}
The properties of the intersection algebra of two principal monomial ideals in a polynomial ring are investigated in detail. Results are obtained regarding the Hilbert series and the canonical ideal of the intersection algebra using methods from the theory of diophantine linear equations with integer coefficients. 

\end{abstract}
\maketitle
\markboth{Intersection algebras for principal monomial ideals in polynomial rings}{Intersection algebras for principal monomial ideals in polynomial rings}

\section{Introduction}

The intersection algebra of two ideals $I, J$ in a commutative Noetherian ring $R$ is a natural concept which is presently poorly understood. This algebra has been considered in a few papers \cite{Ciuperca, Fieldsthesis, Fields}, but not much is known about it in general. The finite generation of this $R$-algebra has been studied in some important cases by J. B. Fields and, more recently, by the second author of this paper in \cite{Malec}. 
Fields has shown that the finite generation of the intersection algebra implies that the function $f(r,s)=\length_R({\rm Tor}_1(R/I^r, R/J^s))$ is a quasi polynomial  in the case of ideals $I, J$ in a local ring $R$ such that $I+J$ is primary to the maximal ideal of $R$.

The important cases considered by Fields and, respectively Malec, in the study of the finite generation of the intersection algebra are these of monomial ideals in polynomial rings with finitely many indeterminates over a field. If one goes beyond this situation, finite generation can fail even in regular local rings as shown by Fields. However, even in the special case of monomial ideals, no other properties of the intersection algebra have been explored. Our paper will be a first systematic attempt at understanding the intersection algebra. 

In this article, we will highlight some interesting properties of the intersection algebra in the first natural situation to consider. Our paper will study in detail the case when $R=k[x_1, \ldots, x_n]$ is a polynomial ring in finitely many variables over a field, and $I, J$ are principal monomial ideals. This is the same case that Malec has considered in \cite{Malec} and it is already complicated enough to stand on its own, as our paper will show. In addition, it offers the possibility of an algorithmic approach that could be useful in understanding the case of arbitrary monomial ideals. The reader should be aware that even the case of principal monomial ideals in $k[x]$, where $k$ a field, is nontrivial and can be used as a good illustrating example for our study. 

We will start with the definition of the intersection algebra. Throughout this paper, $R$ will be a commutative Noetherian ring. Also, bold letters will indicate $n$-tuples, hence $\mathbf{x}^{\boldsymbol \beta}=x_1^{\beta_1} \cdots x_n^{\beta_n}$.\\

\begin{Def}Let $R$ be a commutative Noetherian ring with two ideals $I$ and $J$. Then the \emph{intersection algebra of $I$ and $J$} is $\mathcal{B}=\bigoplus_{r,s \in \mathbb{N}}I^r\cap J^s$. If we introduce two indexing variables $u$ and $v$, then $\mathcal{B}_R(I,J)=\sum_{r,s \in \mathbb{N}}(I^r\cap J^s)u^rv^s\subseteq R[u,v]$. When $R,I$ and $J$ are clear from context, we will simply denote this as $\mathcal{B}$. In fact, we will also drop the parentheses from our notation, since there is no danger of confusion, so $\mathcal{B}_R(I,J)=\sum_{r,s \in \mathbb{N}}I^r\cap J^su^rv^s$. We will often think of $\mathcal{B}$ as a subring of $R[u,v]$, where there is a natural $\mathbb{N}^2$-grading on monomials $b \in \mathcal{B}$ given by $\textrm{deg}(b)=(r,s)\in \mathbb{N}^2$. If this algebra is finitely generated over $R$, we say that $I$ and $J$ have \emph{finite intersection algebra}. 
\end{Def}

\begin{Thm}[Fields~\cite{Fields}] Let $R$ be a Noetherian ring, and $I$ and $J$ monomial ideals in $A=R[x_1, \ldots, x_n]$. Then $I$ and $J$ have finite intersection algebra.
\end{Thm}

It is should be remarked that, when $\B$ is finitely generated over $R$,  its Krull dimension is likely known to the experts. We include a proof for the benefit of the reader.

\begin{Prop}
Let $R$ be a Noetherian domain of dimension $n$ and $I, J$ two nonzero ideals in $R$. If $\B$ is finitely generated as an $R$-algebra, then it has dimension $n+2$.

\end{Prop}

\begin{proof} Let $Q$ be a prime ideal in $\B$ and let $P=Q \cap R$ be its restriction to $R$. Then the dimension inequality \cite{Huneke} says that  
\[\textrm{ht}Q+\textrm{tr.deg}_{\kappa(P)}\kappa(Q) \leq \textrm{ht}P + \textrm{tr.deg}_R\B,\]
where $\kappa(P)$ and $\kappa(Q)$ denote the field of fractions of $R/P$ and $R/Q$, respectively. Then since $\B$ is a domain, both $\frac{u}{1}, \frac{v}{1} \neq 0$ in the fraction field of $\B$, so $\{u,v\}$ form a transcendence basis for $\B$ over $R$. Thus $\textrm{tr.deg}_R\B=2$, and since $\dim R=n$, $\textrm{ht} P \leq n$. So $\textrm{ht}Q \leq n+2$, and thus $\dim \B \leq n+2$.

 Define the following ideal
\[\B_+=\{b \in \B \subset R[u,v]|b \textrm{ has no constant term}\},\]
and consider the localization $\B_{\B_+}$. Note that $u,v \in \B_{\B_+}$, since $u=\frac{Iu}{I}$ and $I \notin \B_+$, and $(u,v)\B_{\B_+}=\B_+\B_{\B_+}=\mathfrak{m}$, the maximal ideal in $\B_{\B_+}$. Since $\B$ is a domain, $\B_{\B_+}$ is too. We claim $\dim \B_{\B_+}\geq2$.\\

Assume that $\dim \B_{\B_+} =1$. So since $0 \neq u \in (u,v)\B_{\B_+}$, ht$(u)=1$. So ht$P=1$ for every $P \in \textrm{Min}(\B_{\B_+}/(u)\B_{\B_+})$, and in fact every prime ideal in $\B_{\B_+}$ has height 1, since $\B_{\B_+}$ is local of dimension 1.\\
So we have a chain
\[0 \subset (u) \subset P = \mathfrak{m},\] and therefore $\sqrt{(u)}=\mathfrak{m}$. So there exists an $n \in \N$ such that $\mathfrak{m}^n \subset (u)$. But $v \in \mathfrak{m}$, so $v^n \in (u)\B_{\B_+}$, which implies that there exists a $b \in \B$,$z \notin \B_+$ such that $v^n=u \frac{b}{z}$. But then $zv^n =ub$, and $u$ does not divide $v^n$, so $u$ must divide $z$. But this is false, because $z$ has a nonzero constant term $z' \in R$, and $u$ can not divide $z'$. \\
So $\dim \B_{\B_+} \geq n+ 2$ and therefore ht $\B_+ \geq n+ 2$. 

Since $\B/\B_+ \cong R$, any chain of primes in $R$ can be extended by 2 primes to a chain in $\B$, and $\dim \B \geq n+2$. \end{proof}

\begin{Def}A \emph{polyhedral cone} $C$ in $\mathbb{R}^d$ is the intersection of finitely many closed linear half-spaces in $\mathbb{R}^d$, each of whose bounding hyperplanes contains the origin. A hyperplane $H$ containing the origin is called a \emph{supporting hyperplane} if $H \cap C \neq 0$ and $C$ is contained in one of the closed half-spaces determined by $H$. If $H$ is a supporting hyperplane of $C$, then $H \cap C$ is called a \emph{face} of $C$.
 Every polyhedral cone $C$ is finitely generated, i.e. there exist $\boldsymbol{c_1}, \ldots, \boldsymbol{c_r} \in \mathbb{R}^d$ with \[C=\{\lambda_1\boldsymbol{c_1}+\cdots+\lambda_r\boldsymbol{c_r}|\lambda_1, \ldots, \lambda_r \in \mathbb{R}_{\geq 0}\}.\] We call the cone $C$ \emph{rational} if $\boldsymbol{c_1}, \ldots, \boldsymbol{c_r}$ can be chosen to have rational coordinates, and $C$ is \emph{pointed} if $C \cap(-C)=\{\boldsymbol{0}\}$.\end{Def}

A special kind of collection of cones are called fans.
\begin{Def}A \emph{fan} is a collection $\Sigma$ of cones $\{C_i\}_{i \in I}$, where $I$ is a finite set, the faces of each $C_i \in \Sigma$ are also in $\Sigma$, and the intersection of every pair of cones in $\Sigma$ is a common face of both of them, if nonempty.
\end{Def}

One major and well known property of these cones that we will use is that they are finitely generated.
 
\begin{Thm} \textup{(Proposition 7.15 in \cite{Miller})} Any pointed affine semigroup $Q$ has a unique finite minimal generating set $\mathcal{H}_Q$.\end{Thm}

 \begin{Def} \textup{(Definition 7.17 in \cite{Miller})} Let $C$ be a rational pointed cone in $\mathbb{R}^d$, and let $Q=C \cap \mathbb{Z}^d$. Then the unique finite minimal generating set $\mathcal{H}_Q$ is called the \emph{Hilbert Basis of the cone $C$}.\end{Def}
  
For any two strings of numbers
\[\mathbf a=(a_1, \ldots, a_n), \mathbf b=(b_1, \ldots, b_n)\textrm{ with }a_i,b_i \in \mathbb{N},\]
we can associate to them a fan of pointed, rational cones in $\mathbb{N}^2$.  Note that in this paper $\mathbb{N}$ denotes the set of nonnegative integers.

\begin{Def} 
\label{fan}We will call two such strings of numbers $(\mathbf{a,b})$ \emph{fan ordered} if 
\[\frac{a_i}{b_i} \geq \frac{a_{i+1}}{b_{i+1}}\textrm{ for all } i=1, \ldots, n-1.\] 

By convention, if $b_i=0$, we will say that $\frac{a_i}{b_i}=\infty$.
Assume $(\mathbf{a}, \mathbf{b})$ are fan ordered. Additionally, let $a_{n+1}=b_0=0$ and $a_0=b_{n+1}=1$. Then for all $i=0, \ldots, n$, let \[C_i=\{\lambda_1 (b_i,a_i )+\lambda_2(b_{i+1},a_{i+1})|\lambda_i \in \mathbb{R}_{\geq 0}\}.\] Let $\Sigma_{\mathbf{a}, \mathbf{b}}$ be the fan formed by these cones and their faces, and call it the \emph{fan of $\mathbf{a}$ and $\mathbf{b}$ in $\N^2$}. Hence \[\Sigma_{\mathbf{a},\mathbf{b}}=\{C_i|i=0, \ldots, n\}.\]

Moreover, $(\mathbf{a},\mathbf{b})$ is called \emph{non-degenerate} if \[\frac{a_i}{b_i} > \frac{a_{i+1}}{b_{i+1}} \textrm{ for all } i=0, \ldots, n.\] 

\end{Def}

Then, since each $C_i$ is a pointed rational cone, $Q_i=C_i \cap \mathbb{Z}^2$ has a Hilbert Basis, say 
\[\cH_i= \mathcal{H}_{Q_i}=\{(r_{i1},s_{i1}), \ldots, (r_{in_i},s_{in_i})\}.\]

Note that any $\Sigma_{\mathbf{a}, \mathbf{b}}$ partitions all of the first quadrant of $\R^2$ into cones, so the collection $\{Q_i|i=0, \ldots, n\}$ partitions all of $\N^2$ as well, so for any $(r,s) \in \N^2$, $(r,s) \in Q_i$ for some $i=0, \ldots, n$.\\

Let $R$ be a UFD with principal ideals $I=(p_1^{a_1}\cdots p_n^{a_n})$ and  $J=(p_1^{b_1}\cdots p_n^{b_n})$, where $p_i, i=1, \ldots, n$ are irreducible elements. For the purposes of this paper, we will assume without loss of generality that the exponent vectors $\mathbf{a}=(a_1, \ldots, a_n)$ and $\mathbf{b}=(b_1, \ldots, b_n)$ are fan ordered. This is because one can index the indeterminates so that the ratios of the entries in $\mathbf{a}, \mathbf{b}$ are non-increasing. This assumption together with the notations introduced in Definition~\ref{fan} will be used throughout the paper.\\

When $R$ is a UFD and $I$ and $J$ are principal ideals, then the intersection $I^r \cap J^s$ can be computed in terms of the irreducible decomposition of the generators of $I$ and $J$. This observation is behind the main result on the finite generation of the intersection algebra proved in~\cite{Malec} and stated below.

 \begin{Thm} [Malec] \label{UFD} Let $R$ be a UFD with principal ideals $I=(p_1^{a_1}\cdots p_n^{a_n})$ and  $J=(p_1^{b_1}\cdots p_n^{b_n})$, where $p_i, i=1, \ldots, n$ are irreducible elements, and let $\Sigma_{\mathbf{a},\mathbf{b}}$ be the fan associated to $\mathbf{a}=(a_1, \ldots, a_n)$ and $\mathbf{b}=(b_1, \ldots, b_n)$. Then $\mathcal{B}$ is generated over $R$ by the set \[\{p_1^{a_1r_{ij}}\cdots p_i^{a_ir_{ij}}p_{i+1}^{b_{i+1}s_{ij}} \cdots p_n^{b_ns_{ij}}u^{r_{ij}}v^{s_{ij}}| i=0, \ldots, n, j=1, \ldots, n_i\},\] where $(r_{ij},s_{ij})$ run over the Hilbert basis for each $Q_i=C_i \cap \Z^2$ for every $C_i \in \Sigma_{\mathbf{a},\mathbf{b}}$.
 \end{Thm}

\begin{Exa}\label{Ex1}
Let $I=(x^5)$ and $J=(x^2)$ in $R[x]$. Then $a_1=5$ and $b_1=2$. Additionally, we let $a_0=b_2=1$ and $a_2=b_0=0$. Then we have that $a_0/b_0>a_1/b_1>a_2/b_2$, 
and we form two cones in $\mathbb{R}^2$:
\[C_0=\{\lambda_1 (0,1 )+\lambda_2(2,5)|\lambda_i \in \mathbb{R}_{\geq 0}\} \textrm{ and } C_1=\{\lambda_1 (2,5)+\lambda_2(1,0)|\lambda_i \in \mathbb{R}_{\geq 0}\}
.\]
Intersecting both cones with $\mathbb{N}^2$ gives two semigroups, $Q_0$ and $Q_1$. Using Macaulay2, we obtain two Hilbert bases
\[\mathcal{H}_{0}=\{(0,1),(1,3),(2,5)\} \textrm{ and } \mathcal{H}_{1}=\{(1,0),(1,1),(1,2),(2,5)\}.\]
Each Hilbert basis gives rise to a set of generators as follows.\\

From $Q_0$:\\

$\begin{aligned}
(0,1): &\left(I^0 \cap J^1\right) u^0v^1=x^2v\\
(1,3): &\left(I^1 \cap J^3\right) u^1v^3=x^6uv^3\\
(2,5): & \left(I^2 \cap J^5 \right)u^2v^5=x^{10}u^2v^5\\
\end{aligned}$\\

From $Q_1$:\\

$\begin{aligned}
(1,0): &\left(I^1 \cap J^0\right) u^1v^0=x^5u\\
(1,1): & \left(I^1 \cap J^1 \right) u^1v^1=x^5uv\\
(1,2): & \left(I^1 \cap J^2 \right)u^1v^2=x^5uv^2\\
(2,5): & \left(I^2 \cap J^5 \right)u^2v^5=x^{10}u^2v^5\\
\end{aligned}$\\

Notice that each generator has the form given in the above theorem: given an $(r,s)$ in $H_0$, the generator is of the form $x^{b_1s}u^rv^s$, and given $(r,s)$ in $H_1$, the generator is of the form $x^{a_1r}u^rv^s$.

So $\mathcal{B}$ is generated over $R$ by $\{x^2v, x^6uv^3, x^{10}u^2v^5, x^5u, x^5uv, x^5uv^2 \}$.
\end{Exa}

\begin{Thm}Using the notations above, the generating set in Theorem~\ref{UFD} is minimal, in that no generator is a product of the others.\end{Thm}
\begin{proof} 
First an easier case: Say $I=(p^a)$ and $J=(p^b)$, where $p$ is an irreducible in $R$, and say that one generator $p^{\max(ar,bs)}u^rv^s$ is a product of the others, in other words,
\begin{equation}\label{decomp}p^{\max(ar,bs)}u^rv^s=\prod_i\left(p^{\max(ar'_i,bs'_i)}u^{r'_i}v^{s'_i}\right)^{c_i}\end{equation}
where $(r'_i,s'_i)$ are elements of $\mathcal{H}_0 \cup \mathcal{H}_1$ and $c_i \in \N$.

We collect all the $(r'_i,s'_i)$ from $Q_0$ into one pair $(r_0,s_0)$, and those from $Q_1$ into another, $(r_1,s_1)$. However, since both $Q_0$ and $Q_1$ contain any points on the face separating them, we must make this partition well defined. Define the two sets
\[\Lambda_0=\{i|(r'_i,s'_i) \in \cH_0\} \textrm{ and } \Lambda_1=\{i|(r'_i,s'_i) \in \cH_1\setminus \cH_0\},\]

and define
\[(r_0,s_0)=\sum_{i \in \Lambda_0}c_i(r'_i,s'_i) \textrm{ and } (r_1,s_1)=\sum_{i \in \Lambda_1}c_i(r'_i,s'_i).\]
In other words, $(r_0,s_0)$ is the sum of all Hilbert basis elements present in the decomposition of $(r,s)$ (including coefficients) that come from the cone $C_0$, and $(r_1,s_1)$ is the sum of all Hilbert basis elements in the decomposition of $(r,s)$ (including coefficients) that come from $C_1$, not including the face between $C_1$ and $C_0$. Note that if the decomposition of $(r,s)$ doesn't contain an element from the cone $C_k$, set $(r_k,s_k)=(0,0)$.
Then by \ref{decomp},
\begin{equation}\label{r}(r,s)=(r_0,s_0)+(r_1,s_1)\end{equation}
and
\begin{equation}\label{max}\max(ar,bs)=\max(ar_0,bs_0)+\max(ar_1,bs_1).\end{equation}

Since $(r_0,s_0) \in Q_0$, $\max(ar_0,bs_0)=bs_0$. Similarly, $\max(ar_1,bs_1)=ar_1$.

Assume $(r,s) \in Q_0$. Then $\max(ar,bs)=bs$, and by \ref{max} and \ref{r},
\[bs=bs_0+ar_1=bs_0+bs_1.\]
Therefore, $ar_1=bs_1$, so $a/b=s_1/r_1$, i.e. $(r_1,s_1)$ lies on the face between $Q_0$ and $Q_1$, which contradicts the definition of $(r_1,s_1)$.

So $(r,s) \in Q_1$ and $\max(ar,bs)=ar$. By \ref{max} and \ref{r},
\[ar=bs_0+ar_1=ar_0+ar_1.\]
So $bs_0=ar_0$, and thus $(r_0,s_0)$ lies on the face between $Q_0$ and $Q_1$. But this means $(r_0,s_0) \in Q_1$, and also $(r_1,s_1)$ and $(r,s) \in Q_1$. But $(r,s)$ is a Hilbert basis element of $Q_1$, and $(r,s)=(r_0,s_0)+(r_1,s_1)$, which is a contradiction. \\

The proof of the general case is similar. Let $p_1, \ldots, p_n$ be irreducibles in $R$, and $I=(p_1^{a_1} \ldots p_n^{a_n}), J=(p_1^{b_1} \ldots p_n^{b_n})$. Say that one generator is a product of the others, i.e.
\begin{equation}\label{bigdecomp}\prod_{i=1}^m(p_1^{\max(a_1r'_i,b_1s'_i)}\cdots p_n^{\max(a_nr'_i,b_ns'_i)})^{c_i}=p_1^{\max(a_1r,b_1s)}\cdots p_n^{\max(a_nr,b_ns)},\end{equation}
where $(r'_i,s'_i)$, $i=1, \ldots, m$ and $(r,s)$ are Hilbert basis elements, and $c_i \in \N$. 

We claim we can assume that $b_i \neq 0$ for all $i=1, \ldots, n$. To see this, note that if there exists an $h$ such that $b_h=0$, then $b_1, \ldots, b_{h-1}=0$ by the fan-ordering. Also, if $b_1=0$, then $a_1 \neq 0$ (otherwise, $p_1$ would simply not be in the decompositions of the generators of $I$ and $J$). So
\begin{equation*}\begin{aligned}a_1r=\sum_{i=1}^mc_ia_1r_i' \textrm{ implies } r=\sum_{i=1}^mc_ir_i'.
\end{aligned}\end{equation*}

So by canceling in \ref{bigdecomp},
\begin{equation*}\prod_{i=2}^m(p_2^{\max(a_2r'_i,b_1s'_i)}\cdots p_n^{\max(a_nr'_i,b_ns'_i)})^{c_i}=p_2^{\max(a_2r,b_2s)}\cdots p_n^{\max(a_nr,b_ns)}.\end{equation*}

If $b_2=0$, we continue in the same way with canceling the $p_2$ terms in \ref{bigdecomp}, until the first nonzero $b$, say $b_h$. Then we have \ref{bigdecomp} with only terms $p_h, \ldots, p_n$, and $b_h, \ldots, b_n$ are all nonzero.

Now assume all $b_i \neq 0$.

We again partition all the $(r'_i,s'_i)$, and sum them into $n+1$ pairs $(r_i,s_i), i=0, \ldots, n$. To make this partition well defined, define the $n+1$ sets
\[ \Lambda_0=\{i|(r'_i,s'_i) \in \cH_0\} \textrm{ and } \Lambda_k=\{i|(r'_i,s'_i) \in \cH_k \setminus \cH_{k-1}\} \textrm{ for all }k=1, \ldots, n\}.\]

Then define
\[(r_k,s_k)=\sum_{i \in \Lambda_k}c_i(r'_i,s'_i).\]

By convention, if $(b_i,a_i)=k(b_{i+1},a_{i+1})$ for some $k \geq 0$, rational, and some $i$ from 0 to $n$, we say that $(r_i,s_i)=(0,0)$.

Then by \ref{bigdecomp},
\begin{equation}\label{rs}(r,s)=\sum_{i=0}^n(r_i,s_i)\end{equation}
and
\begin{equation}\label{bigmax}\max(ar,bs)=\sum_{i=0}^n\max(ar_i,bs_i).\end{equation}

Assume that $(r,s) \in Q_j$, so by the cone structure
\[p_1^{\max(a_1r,b_1s)}\cdots p_n^{\max(a_nr,b_ns)}=p_1^{a_1r}\cdots p_j^{a_jr}p_{j+1}^{b_{j+1}s}\cdots p_n^{b_ns}.\]
Therefore
\begin{align*}a_1r&=\sum_{i=0}^n \max(a_1r_i,b_1s_i)\\
\vdots \quad &=\quad \vdots\\
a_jr&=\sum_{i=0}^n\max(a_jr_i,b_js_i)\\
b_{j+1}s&=\sum_{i=0}^n\max(a_{j+1}r_i,b_{j+1}s_i)\\
\vdots \quad &= \quad \vdots\\
b_{n}s&=\sum_{i=0}^n\max(a_{n}r_i,b_{n}s_i).\\
\end{align*}

Also, since every nonzero $(r_k,s_k)$ is in $Q_k \setminus Q_{k-1}$, for all $k=0, \ldots n$, 
\begin{align*}a_ir_k &\geq b_is_k &\textrm{ for all }i < k\\
a_kr_k&>b_ks_k&\\
a_{i}r_{k}&\leq b_{i}s_{k} &\textrm{ for all }i >k.\\
\end{align*}

Therefore, the above sums become
\[\begin{array}{ccccc}a_1r &=& b_1\sum_{i=0}^0s_i&+&a_1\sum_{i=1}^n r_i\\
a_2r&=&b_2\sum_{i=0}^1s_i&+&a_2\sum_{i=2}^nr_i\\
\vdots \quad &=&\quad \vdots&&\vdots\\
a_jr&=& b_j\sum_{i=0}^{j-1}s_i&+& a_j\sum_{i=j}^nr_i\\
b_{j+1}s&=&b_{j+1}\sum_{i=0}^js_i&+&a_{j+1}\sum_{i=j+1}^nr_i\\
\vdots \quad &=& \quad \vdots&&\vdots\\
b_{n}s&=&b_n\sum_{i=0}^{n-1}s_i&+&a_{n}\sum_{i=n}^nr_i.\\
\end{array}\]

By (\ref{rs}), we can again clear terms on both sides to obtain
\[\begin{array}{ccc}a_1\sum_{i=0}^0r_i &=& b_1\sum_{i=0}^0s_i\\
a_2\sum_{i=0}^1r_i&=&b_2\sum_{i=0}^1s_i\\
\vdots \quad &=&\quad \vdots\\
a_j\sum_{i=0}^{j-1}r_i&=& b_j\sum_{i=0}^{j-1}s_i\\
b_{j+1}\sum_{i=j+1}^ns_i&=&a_{j+1}\sum_{i=j+1}^nr_i\\
\vdots \quad &=& \quad \vdots\\
b_{n}\sum_{i=n}^n s_i&=&a_{n}\sum_{i=n}^nr_i.\\
\end{array}\]

By the last equation in the above collection, $b_ns_n=a_nr_n$. So $(r_n,s_n)$ lies on the line with slope $a_n/b_n$. But this is the face in between $Q_n$ and $Q_{n-1}$, contradicting the definition of $(r_n,s_n)$. So there are no generators coming from $Q_n$, so $(r_n,s_n)=(0,0)$.

From the $n-1$th equation in the collection,
$b_{n-1}s_{n-1}+b_{n-1}s_n=a_{n-1}r_{n-1}+a_{n-1}r_n$. But since $(r_n,s_n)=(0,0)$, $a_{n-1}/b_{n-1}=s_{n-1}/r_{n-1}$. So $(r_{n-1},s_{n-1})$ lies on the face in between $Q_{n-1}$ and $Q_{n-2}$, contracting the definition of $(r_{n-1},s_{n-1})$. So again, $(r_{n-1},s_{n-1})=0$. Continuing in this way, we see that $(r_k,s_k)=(0,0)$ for all $j+1 \leq k \leq n$.

By the first equation in the list, $a_1r_0=b_1s_0$, so $(r_0,s_0)$ is on the line between $Q_0$ and $Q_1$.

The second equation says $a_2r_0+a_2r_1=b_2s_0+b_2s_1$. But since $(r_0,s_0) \in Q_0$ and $(r_1,s_1) \in Q_1$, $a_2r_0 \leq b_2s_0$ and $a_2r_1 \leq b_2s_1$. Therefore $a_2r_0=b_2s_0$ and $a_2r_1=b_2s_1$.

This together with the fact that $a_1r_0=b_1s_0$ gives
\[\frac{a_2}{b_2}=\frac{s_1}{r_1}=\frac{s_0}{r_0}=\frac{a_1}{b_1}.\]
Therefore, by convention, $(r_1,s_1)=(0,0)$.

In a similar way, the third equation $a_3(r_0+r_2)=b_3(s_0+s_2)$, together with $a_3r_0 \leq b_3s_0$ and $a_3r_2 \leq b_3s_2$ implies that $\frac{a_3}{b_3}=\frac{a_2}{b_2}$, and so by convention $(r_2,s_2)=(0,0)$. Continuing in this fashion shows that $(r_k,s_k)=(0,0)$ for all $1 \leq k \leq j$.

Therefore, $(r,s)=(r_0,s_0)$. Since $(r,s)$ was chosen to be a Hilbert basis element, $(r_0,s_0)$ is too. Therefore $(r,s)$ is not a sum of multiple Hilbert basis elements, and therefore $p_1^{\max(a_1r,b_1s)}\cdots p_n^{\max(a_nr,b_ns)}$ is not a product of multiple algebra generators.\end{proof}

Let $k$ be a field and $Q$ a semigroup. We denote $k[Q]$ the semigroup ring associated to $Q$, namely the $k$-algebra with $k$-basis $\{t^a | a \in Q\}$ and multiplication defined by $t^a \cdot t^b=t^{a+b}$. Also note our notation: when $x$ is a homogeneous element in a semigroup ring, $\log(x)$ denotes its exponent vector, and if $X$ is a collection of homogeneous elements, $\log(X)$ refers to the set of exponent vectors of all the monomials in $X$.

\begin{Prop}[Theorem 3.5 in \cite{Malec}] If $R$ is a polynomial ring in $n$ variables over $k$, and $I$ and $J$ are ideals generated by monomials (i.e. monic products of variables) in $R$, then $\mathcal{B}$ is a semigroup ring.\end{Prop}

We would like to discuss a couple of important properties of the intersection algebras. We will state them in a more general context, that of fan algebras.

\begin{Def}[Definition 4.1 in \cite{Malec}] Given a fan of cones $\Sigma=\Sigma_{\mathbf{a}, \mathbf{b}}$, a function $f:\N^2 \rightarrow \N$ is called \emph{fan-linear} if 
\begin{enumerate}
\item
$f$ is nonnegative and linear on each subgroup $Q_i=C_i \cap \Z^2$ for each $C_i \in \Sigma_{\mathbf{a}, \mathbf{b}}$, and 
\item
subadditive on all of $\N^2$, i.e.

\[f(r,s)+f(r',s') \geq f(r+r', s+s') \textrm{ for all }(r,s), (r',s') \in \N^2.\]
\end{enumerate}
In other words, the condition $(1)$ can be rephrased as follows:
for $i=0, \ldots, n$, there exists a function $g_i$ linear on $C_i \cap \Z^2$ such that
\[f(r,s) = g_{i}(r,s)  \textrm{ when } (r,s) \in C_i\cap \Z^2 \textrm{ for each } i=0, \ldots n,\] 
and $g_i=g_j$ for every $(r,s) \in C_i \cap C_j\cap \Z^2$.
 
 \medskip
 \noindent
 Then we define
\[\B(\Sigma,f)=\bigoplus_{r, s \geq 0} I_1^{f_1(r,s)}\cdots I_n^{f_n(r,s)}u_1^ru_2^s\]
to be the \emph{fan algebra of $f$ on $\Sigma$}, where $f=(f_1, \ldots, f_n)$.

\end{Def}

\begin{Exa} Let $\mathbf{a}=\{1\}=\mathbf{b}$, so $\Sigma_{\mathbf{a}, \mathbf{b}}$ is the fan defined by 
\begin{align*}C_0&=\{\lambda_1(0,1)+\lambda_2(1,1)|\lambda_i \in \mathbb{R}_{\geq 0}\}\\
C_1&=\{\lambda_1(1,1)+\lambda_2(1,0)|\lambda_i \in \mathbb{R}_{\geq 0}\},\end{align*}
and set $Q_i=C_i \cap \Z^2$.
Also let \[f=\left\{ \begin{array}{ll}
g_0(r,s)=r+2s & \textrm{if $(r,s) \in Q_0$}\\ g_1(r,s)=2r+s & \textrm{if $(r,s) \in Q_1$}\end{array} \right.\]
Then $f$ is a fan-linear function as it was checked in Example 4.2 in~\cite{Malec}.\end{Exa}

Our important property often shared by classes of semigroup rings is normality. It turns out that this property is also shared by fan algebras as it is shown below. For an introduction to normality in the context of semigroup rings, we refer to the standard reference ~\cite{Bruns}, Chapter 6, Section 1.

 \begin{Thm}\label{normal}Let $R$ be the $n$-dimensional polynomial ring over a field, and $I_1, \ldots, I_n$ be principal monomial ideals of $R$. Also, let $\Sigma_{\mathbf{a}, \mathbf{b}}$ be a fan of cones in $\N^2$, with fan-linear functions $f=f_1, \ldots, f_n$. Then $\B(\Sigma_{\mathbf{a},\mathbf{b}},f)$ is normal.
\end{Thm}

\begin{proof}Since $\B$ is a fan-algebra of principal monomial ideals over a polynomial ring, $\B$ is a semigroup ring. So $\B=k[Q]$ for the semigroup $Q$ which consists of all the exponent vectors of all the elements of $\B$. Let $z=(z_1, \ldots, z_n,r,s)$, $m \in \N$ and $mz \in Q$. We claim $z \in Q$.\\
 
 Since $Q$ is the semigroup that defines $\B$ as a semigroup ring, 
 \[Q=\{(a_1, \ldots, a_n, r, s)|a_j\geq f_j(r,s) \textrm{ for all } j=1, \ldots n, (r,s) \in \N^2\}.\]
 And since $mz \in Q$, $(mr,ms) \in Q_i$ for some $i$. Then we have the following:
 \begin{equation*}\begin{aligned}
 mz_1 &\geq f_1(mr,ms)\\
 \vdots & \vdots\\
 mz_n& \geq f_n(mr,ms). \\
 \end{aligned}\end{equation*}
 But each of the functions $f_k$ are fan-linear, so on $Q_i$, $f_k(mr,ms)=mf_k(r,s)$. So we have 
 \[mz_k  \geq mf_k(r,s) \textrm{ for all }k=1, \ldots, n\]
 and so
 \[z_k \geq f_k(r,s) \textrm{ for all } k=1, \ldots, n,\]
 hence $z=(z_1, \ldots, z_n,r,s) \in Q$.
 
 Therefore $Q$ is normal, hence $\B=k[Q]$ is normal.
\end{proof}

All normal semigroup rings are Cohen-Macaulay, a well known result of Hochster~\cite{Hochster} (see also Theorem 6.1.3 in~\cite{Bruns}).

\begin{Cor}
\label{CM} Let $\B(\Sigma_{\mathbf{a},\mathbf{b}},f)$ be as defined in Theorem \ref{normal}. Then $\B$ is Cohen-Macaulay.
\end{Cor}

\begin{Rem}
Similar results to Theorem~\ref{normal} and Corollary~\ref{CM} can be formulated for fan algebras associated to a fan of cones on $\mathbb{N}^n, n \in \mathbb{N}$.
\end{Rem}

\section{Approach via Linear Diophantine Equations with Integer Coefficients}

In this section, we explore in detail the relationship between the intersection algebra for principal ideals in the monomial case and semigroup rings, context where we then conduct our study. First, let us apply Theorem~\ref{UFD} to our situation. 

\begin{Thm}
 \label{gener} Let $I=(x_1^{a_1}\cdots x_n^{a_n})$ and  $J=(x_1^{b_1}\cdots x_n^{b_n})$ be principal ideals in $R=k[x_1, \ldots, x_n]$, and let $\Sigma_{\mathbf{a},\mathbf{b}}$ be the fan associated to $\mathbf{a}=(a_1, \ldots, a_n)$ and $\mathbf{b}=(b_1, \ldots, b_n)$. Let 
 \[Q_i=C_i \cap \Z^2 \textrm{ for every } C_i \in \Sigma_{\mathbf{a},\mathbf{b}}\] 
 and $\cH_i$ be its Hilbert basis of cardinality $n_i$ for all $i=0, \ldots, n$.
 Further, let $Q$ be the subsemigroup in $\N^{n+2}$ generated by 
 \[\{(a_1r_{ij}, \ldots, a_ir_{ij},b_{i+1}s_{ij},\ldots,b_ns_{ij},r_{ij},s_{ij})| i=0, \ldots, n, j=1, \ldots, n_i\} \cup \log(x_1, \ldots, x_n),\]
 where $x_i$ are considered as variables in $R[u,v]$ and $(r_{ij},s_{ij})\in \cH_i$ for every $i=0, \ldots n, j=1, \ldots, n_i$. Then $\mathcal{B}=k[Q]$.
 \end{Thm}
 
 \begin{proof}
 Since a polynomial ring $R=k[x_1, \ldots, x_n]$ is a UFD, Theorem~\ref{UFD} concludes that the intersection algebras of principal ideals in polynomial rings are finitely generated over $R$. Then, since $R$ is generated over $k$ by the variables $x_1, \ldots, x_n$, adding those variables to our list of generators will provide a generating set for $\B$ over $k$.
 
 \end{proof}
 
 One can hope that the result stated above, which describes a canonical set of generators of the intersection algebra, will help greatly in its study, and indeed this can be done in particular cases. However, there is one serious difficulty. It is notoriously difficult to explicitly compute the generators of the semigroups $Q_i$, $i =0, \ldots, n$ that appear in the result above. They can be theoretically obtained by using classical results of van der Corput \cite{vdC2, vdC1}, however in the form provided by these results, they cannot be readily applied to understand the properties of $\B$ from a theoretical perspective. For example, studying the presentation ideal of $\B$ as a $k$-algebra seems to be a very difficult task, despite the fact that its computation of the algebra generators, and therefore the relations among them, can be performed in any given particular case by using a computer algebra system as shown by Malec in~\cite{Malec}. 
 
Therefore we will present an approach to understanding some of the algebraic properties of $\B(I,J)$ where $I=(x_1^{a_1}\cdots x_n^{a_n})$ and  $J=(x_1^{b_1}\cdots x_n^{b_n})$ are principal ideals in $R=k[x_1, \ldots, x_n]$ that avoids this problem. The approach is based upon work of Stanley in~\cite{Stanley} which will be reviewed shortly.

\medskip
\noindent
From now on, we will assume that $(\mathbf{a},\mathbf{b})$ are fan ordered, however at times we will restate this hypothesis to reinforce it. As remarked earlier, this assumption is not limiting.

First, we need recall some facts from semigroup rings associated to linear diophantine equations with integer coefficients from Chapter I, Section 3 of \cite{Stanley}.
\begin{Def}Let $\Phi$ be an $r \times n$ $\Z$-matrix, $r\leq n$, and rank $\Phi=r$. Define
\[E_{\Phi}:=\{\boldsymbol\beta \in \N^n|\Phi\boldsymbol\beta=0\} \subset {\rm Ker}(\phi).\]
Then $E_{\Phi}$ is clearly a subsemigroup of $\N^n$.\\ Let $R_\Phi:=kE_\Phi$, the semigroup algebra of $E_\Phi$ over $k$. We identify $\boldsymbol\beta \in E_\Phi$ with $x^{\boldsymbol\beta}$, so that $R_\Phi \subseteq k[x_1, \ldots, x_n]$ as a subalgebra graded by monomials.
\end{Def}

To translate our problem into this language, we must describe all monomials in $\B$ as having exponents that are solutions to a system of equations. Since any monomial must have the form $(\prod_i x_i^{m_i}) u^rv^s$, with $m_i\geq \max(a_i r,b_is)$, there exists an $h_i,k_i \in \N$ such that the log of any monomial must satisfy
\[m_i=a_ir+h_i=b_is+k_i, i=1, \ldots, n.\]
Now let
\[\Phi=\Phi_{a,b}=\left(\begin{array}{cccccccccc}
a_1 & 0 & 1 & 0 & -1 &0&0&0& ...&0\\
 0 & b_1 & 0 & 1 & -1 &0&0&0& ...&0\\
 a_2 & 0 &  0 & 0 & 0 &1&0&-1& ...&0\\
  0 & b_2 & 0 & 0 & 0 &0&1&-1& ...&0 \\
  \cdots & \cdots & \cdots & \cdots & \cdots & \cdots & \cdots & \cdots & \cdots \\
 \end{array}\right),\] and $\boldsymbol\beta = (r,s, h_1, k_1, m_1, \dots, h_n, k_n,m_n) \in \N^{3n+2}.$

Then $E_\Phi=E_{\Phi_{a,b}}=\{\boldsymbol\beta \in \N^{3n+2}|\Phi\boldsymbol\beta=0\}$. Then $E_\Phi$ is a subsemigroup in $\N^{3n+2}$, and to recover our semigroup $Q$ where $\B=k[Q]$, we project $E_\Phi$ onto $\N^{n+2}$ by "forgetting" all $h_i$ and $k_i$. This establishes and isomorphism between $Q$ and $E_\Phi$, and so $\B$ and $R_\Phi$ are as well.

Let us define the following additive maps $\pi_i : \mathbb{N}^{3n+2} \to \mathbb{N}^5$, $\pi_i (r,s, h_1, k_1, m_1, \dots, h_n, k_n, m_n) = (r, s, h_i, k_i,m_i)$, for all $i = 1, \ldots, n$.

\begin{Exa}
We revisit the previous example in this new context: recall that $I=(x^5)$ and $J=(x^2)$, so $a_1=5$ and $b_1=2$. Therefore any monomial in $\B$ is of the form $x^{m_1}u^rv^s$ with $m_1 \geq \max(5r,2s)$, and there exists an $h_1,k_1 \in \N$ such that the log of any monomial satisfies the equations
\[m_1=5r+h_1=2s+k_1.\]
Then $\Phi$ becomes
\[ \left(
\begin{tabular}{ccccc}
5 & 0 & 1 & 0 &-1\\
0&2&0&1&-1\\
\end{tabular}\right)\]
and $\boldsymbol\beta=(r,s,h_1,k_1,m_1) \in \N^5$. Then $E_\Phi=\{\boldsymbol\beta \in \N^5|\Phi\boldsymbol\beta=0\}$, and we recover the exponent semigroup $Q$ by ''forgetting" $h_1$ and $k_1$. In particular, $x^6uv^3$ is an element of $\B$, and it corresponds to the vector $(1,3,1,0,6)$ since $5 \cdot 1+1=2 \cdot 3+0=6$. Omitting the 1 and the 0 gives the exponent vector $(3,1,6)$, which is the log of the exponent vector of the original monomial (after reordering). 

Since $n=1$, the map $\pi_1$ is the only projection map, and it is the identity. 
\end{Exa}

Reformulating $\B$ in terms of this matrix allows us to easily prove some more properties of this algebra. This requires a few more definitions and results from \cite{Stanley}.

\begin{Def}(\cite{Stanley})
\label{dfund}
We say that $\boldsymbol\beta \in E_\Phi$ is \emph{fundamental} if $\boldsymbol\beta=\boldsymbol\gamma+\boldsymbol\delta$, $\boldsymbol\gamma, \boldsymbol\delta \in E_\Phi$ implies $\boldsymbol\gamma=\boldsymbol\beta$ or $\boldsymbol\delta=\boldsymbol\beta$. We set
\[\textrm{FUND}_\Phi:=\textrm{ set of fundamental elements of } \Phi,\]and note that it exists since the vectors in $E_{\Phi}$ have nonnegative integer entries.
\end{Def}

It is clear that any set which generates $E_\Phi$ contains $\textrm{FUND}_\Phi$, so \[k[\mathbf{x}^{\boldsymbol\delta}|\boldsymbol\delta \in \textrm{FUND}_\Phi] \subseteq R_\Phi.\] Since $\B = k[Q]$ is finitely generated as a $k$-algebra,  and $E_\Phi$ is isomorphic to $Q$ we see that $E_\Phi$ has finitely many semigroup generators. In particular, $|\textrm{FUND}_\Phi |<\infty$.

\begin{Def}(\cite{Stanley})
\label{dcfund}
We say that $\boldsymbol\beta \in E_\Phi$ is \emph{completely fundamental} if whenever $n >0$ and $n\boldsymbol\beta=\boldsymbol\gamma+\boldsymbol\delta$ for $\boldsymbol\gamma, \boldsymbol\delta \in E_\Phi$, then $\boldsymbol\gamma=n_1\boldsymbol\beta$ and $\boldsymbol\delta=n_2\boldsymbol\beta$ for  some $0 \leq n_1, n_2 \leq n, n_1+n_2=n$. We set
\[\textrm{CF}_\Phi := \textrm{ set of completely fundamental elements of }E_\Phi.\]
\end{Def}

\medskip
\noindent
Since any generating set for $E_\Phi$ contains $\textrm{FUND}_\Phi$, and we already know a generating set for $Q \simeq E_\Phi$ as in Theorem~\ref{gener}, $\textrm{FUND}_\Phi$ must be among the points that we get from our generators. We claim that the two sets are in fact equal.

First, recall a few facts about the construction of $\B(I,J)$ by going back to Theorem~\ref{gener}. The generating set for $\B= k[Q]$ is built from the cones
\[C_i =\{\lambda (b_i, a_i)+\mu(b_{i+1},a_{i+1})|\lambda, \mu\in \R_+\}
\]
and, when intersected with $\Z^2$, they form semigroups, $Q_i$, each with a corresponding Hilbert basis $\mathcal{H}_{i}$, $i=0, \ldots, n$. Each basis element has a corresponding algebra generator depending on which cone it comes from.  For example, let $i$ such that $0 \leq i \leq n$. Then
\[\{(a_1r_{ij}, \ldots, a_ir_{ij},b_{i+1}s_{ij},\ldots,b_ns_{ij},r_{ij},s_{ij})| i=0, \ldots, n, j=1, \ldots, n_i\} , \]
 where $(r_{ij},s_{ij})\in \cH_i$ for every $i=0, \ldots n, j=1, \ldots, n_i$, is a generating set for $Q$. 

 This produces a set of generators for $\B$ in the following way:

$$G_i=\{\left( \prod_{j=1}^nx_j^{m_j}\right) u^{r}v^{s}|(r,s) \in \mathcal{H}_{i}, m_{j}=a_{j}r, 1\leq j \leq i, m_{j} =b_js, i+1 \leq j \leq n \}. $$The monomials in $\cup _{i=0}^n G_i$, together with $x_1, \ldots, x_n$, generate $\B$ over $k$.

\begin{Def}Fix $i$, $1\leq i \leq n$. Let $\boldsymbol\alpha_i$ equal the vector in $\mathbb{N}^{3n+2}$ such that $\pi_i (\boldsymbol\alpha_i) = (0, 0, 1, 1, 1)$ and $\pi_j (\boldsymbol\alpha_i) = (0, 0, 0, 0, 0)$ for all $j \neq i$.  Let $A = \{\boldsymbol\alpha_i : i=1, \ldots, n\}$.

For $(r,s) \in \N^2$, we let $$\boldsymbol\beta(r,s) = (r,s, \ldots, \max(a_jr, b_j s) - a_j r, \max(a_jr, b_j s) -b_js, \max(a_jr, b_js), \ldots).$$ That is, for every $j =1, \ldots, n$,
$$\pi_j (\boldsymbol\beta(r,s)) = (r,s, \max(a_jr, b_j s) - a_j r, \max(a_jr, b_j s) -b_js, \max(a_jr, b_js)).$$
For $(r,s)\in \cH_i$, therefore $\boldsymbol\beta(r,s)$ equals the vector $\boldsymbol\beta$ in $\mathbb{N}^{3n+2}$ such that $$\pi_j (\boldsymbol\beta) = (r,s, 0, a_jr-b_js, a_jr), j\leq i, \pi_l(\boldsymbol\beta) = (r, s,  b_ls-a_lr, 0, b_ls), l \geq i+1.$$ Let $B= \{ \boldsymbol\beta(r,s): (r,s) \in \cH_i, i=0, \ldots n\}.$

The reader should note that the elements of $A \subset E_\Phi$ correspond to the generators $x_1, \ldots, x_n$ of $\B$, while the elements of $B \subset E_\Phi$ correspond to the generators of $\B$ that are obtained from the Hilbert bases of $\cH_i$ as in Theorem~\ref{gener}. 
\end{Def}

\begin{Exa} \label{Ex2}Returning to Example \ref{Ex1}, since $n=1$, $\boldsymbol\alpha_1=(0,0,1,1,1)$ and then $A= \{ (0,0,1,1,1)\}$. There is only one projection map $\pi_1$, so for each $(r,s)$ in $\cH_0,$ $\boldsymbol\beta(r,s)=(r,s,b_1s-a_1r,0,b_1s)$, and for each $(r,s)$ in $\cH_1$, $\boldsymbol\beta(r,s)=(r, s,  b_ls-a_lr, 0, b_ls)$. Since $\cH_0=\{(0,1),(1,3),(2,5)\}$, these $\boldsymbol\beta(r,s)$ are then $\{(0,1,2,0,2),(1,3,1,0,6),(2,5,0,0,10)\}$. Similarly, $\cH_1=\{(1,0),(1,1),(1,2),(5,2)\}$ and $\boldsymbol\beta(r,s)=(r,s,0,a_1r-b_1s,a_1r)$ in $\cH_1$, so these $\boldsymbol\beta(r,s)$ are $\{(1,0,0,5,5),(1,1,0,3,5),(1,2,0,1,5),(2,5,0,0,10)\}$. $B$ is the union of both of these sets. The next theorem shows that the union of $A$ and $B$ is the collection of fundamental elements of $E_\Phi$.
\end{Exa}

\begin{Thm}\label{fund}
Let $\mathbf{a,b} \in \N ^n$, and $\Phi=\Phi_{\mathbf{a,b}}$. Then the set of fundamental elements of $E_\Phi$, $\textrm{FUND}_\Phi$, is  $A \cup B$.

\end{Thm}

\begin{proof} From the definition of $\boldsymbol\alpha_i$, $i=1, \ldots, n$, (namely because entries are only $0$'s and $1$'s) we see that $A$ consists of fundamental elements.

Fix $i$. We want to show that each $\boldsymbol\beta \in B$ is fundamental, so  say that $\boldsymbol\beta=\boldsymbol\gamma+\boldsymbol\delta$. Then for $l \geq i+1$ \[ \pi_l(\boldsymbol\gamma)=(r',s',h',0,m') \textrm{ and } \pi_l(\boldsymbol\delta)=(r^{''},s^{''},h^{''},0,m^{''}),\] where
\[a_lr'+h'=b_ls'=m'\textrm{ and }a_l^{''}+h^{''}=b_ls''=m^{''},\] which implies that $b_ls' \geq a_lr'$ and $b_ls^{''} \geq a_lr^{''}$. Similarly,   for $j \leq i$, $b_js' \leq a_jr'$ and $b_js^{''} \leq a_jr^{''}$so both $(r',s')$ and $(r^{''},s^{''})$ are in $Q_i$.  But $r=r'+r^{''}$ and $s=s'+s^{''}$, and $(r,s)$ is a Hilbert basis element. So either $(r,s)=(r',s')$ or $(r,s)=(r^{''},s^{''})$, and thus either $\boldsymbol\beta=\boldsymbol\delta$ or $\boldsymbol\beta=\boldsymbol\gamma$, and so 
$$A \cup B \subseteq \textrm{FUND}_\Phi.$$

As remarked in Definition~\ref{fund}, any set of generators of $E_\Phi$ must contain $\textrm{FUND}_\Phi$. But the elements in $A \cup B$ correspond to the algebra generators for $\B$ over $k$ and so they are generators for $E_\Phi$. So $\textrm{FUND}_\Phi \subseteq A \cup B$ as well, and this concludes the proof.\\

\end{proof}

We will proceed to determining the completely fundamental elements. Let $\mathbf{a, b} \in \mathbb{N}^{n}$. For $i=0, \ldots, n+1$, we let $p_i, q_i$ be relatively prime, positive integers such that $a_i/b_i = p_i/q_i$, if $b_i \neq 0$. If $b_i =0$ we let $p_i =1, q_i =0$.


\begin{Thm}
\label{CF} 
Let $\mathbf{a,b} \in \N ^n$ and $\Phi=\Phi_{\mathbf{a,b}}$. Then the completely fundamental elements of $E_\Phi$ are \[\textrm{CF}_\Phi=\{\boldsymbol\beta(q_i, p_i), i=0, \ldots n+1\} \cup A.\]
\end{Thm}

\begin{proof}  

Now, $\textrm{CF}_\Phi \subseteq \textrm{FUND}_\Phi$, so to determine the completely fundamental elements, we need only discard those fundamental elements which do not fit the definition. Also, note that for each $=0, \ldots, n+1$, $(q_i, p_i)$ is part of the Hilbert Basis $\cH_i$, so each $\boldsymbol\beta(q_i,p_i)$ is a fundamental element by Theorem~\ref{fund}.

Also, it is clear from the positions of the zeros in each of the elements of A that they are completely fundamental. 

Fix $i$. We will show that $\boldsymbol\beta (q_i, p_i)$ is completely fundamental. So, let us assume that there exists $c >0$ such that $c \boldsymbol\beta (q_i, p_i) = \boldsymbol\gamma_1 + \boldsymbol\gamma _2$. 
Then $$ c (q_i, p_i, 0, 0, q_ia_i) = \pi_i (\boldsymbol\gamma_1) + \pi_i (\boldsymbol\gamma_2).$$ Say $\pi_i (\boldsymbol\gamma_j ) = (r_j, s_j, h_j, k_j, m_j)$ for $j=1,2$. So, $h_j =k_j =0,$ and then $r_j p_i = s_j q_i, j=1,2$. This implies that there exist $0 \leq c_1 \leq c$ such that $(r_1, s_1) = c_1 (q_i, p_i)$ and furthermore
$c_1 \boldsymbol\gamma_1 = \boldsymbol\beta(q_i, p_i).$ 

Let $\boldsymbol\beta \in \textrm{FUND}_\Phi$ that is not in the set $\{\boldsymbol\beta(q_i, p_i), i=0, \ldots n+1\} \cup A$, and say without loss of generality that
\[\boldsymbol\beta=\boldsymbol\beta(r,s), (r,s) \in \cH_i.\]
Then since \[(r,s) \in Q_i = C_i \cap \Z, \textrm{ where } C_i=\{\lambda (b_i, a_i)+ \mu(b_{i+1},a_{i+1}): \lambda, \mu \in \R^+\},\]
we can also write $(r,s)=\lambda (q_i, p_i)+ \mu(q_{i+1},p_{i+1})$ for some $\lambda, \mu \in \R^+$. In fact, since $(r,s) \in  \cH_i$, $\lambda, \mu$ are not simultaneously integers.

Clearing denominators then shows that there exists an integer $c >1$, and integers $(d,e)$  such that \[c(r,s)=d (q_i, p_i)+ e(q_{i+1},p_{i+1}).\]

To conclude, note that 

$$c \boldsymbol\beta(r,s) = \boldsymbol\beta (cr, cs) = \boldsymbol\beta( d (q_i, p_i)+ e(q_{i+1},p_{i+1})) = d \boldsymbol\beta (q_i, p_i) + e\boldsymbol\beta(q_{i+1}, p_{i+1}).$$

So $\boldsymbol\beta$ is not in $\textrm{CF}_\Phi$.
\end{proof}

\begin{Exa} \label{Ex3}  In Example \ref{Ex2}, $a_i$ and $b_i$ are already relatively prime for $i=0, 1, 2$, so $a_i=p_i$ and $b_i=q_i$ for all $i$. So the completely fundamental elements will be $\boldsymbol\beta(0,1)=(0,1,2,0,2),\boldsymbol\beta(2,5)=(2,5,0,0,10), \boldsymbol\beta(1,0)=(1,0,0,5,5)$, as well as $\boldsymbol\alpha=(0,0,1,1,1)$.
\end{Exa}

The completely fundamental elements can be used to obtain information about the Hilbert series of the intersection algebra as it was shown by Stanley in~\cite{Stanley}. The reader should keep in mind that, when $I=(x_1^{a_1}\cdots x_n^{a_n})$ and  $J=(x_1^{b_1}\cdots x_n^{b_n})$ are principal ideals in $R=k[x_1, \ldots, x_n]$, then $\B(I, J)$ is a $\N^{n+2}$-graded $k$-algebra. As remarked before, $\B(I,J)$ is $k$-isomorphic to 
$R_\Phi$ which is $\N^{3n+2}$-graded. Therefore, the computation of the Hilbert series of $R_\Phi$ will produce the Hilbert series of $\B$. The necessary results from \cite{Stanley} for the computation of the Hilbert series of $R_\Phi$ are excerpted below.

\begin{Thm}\textup{(Corollary 3.8 in \cite{Stanley})}
\label{S-HS}The Hilbert series of $R_\Phi$ is $F(R_\Phi, \mathbf{z})=\sum_{\boldsymbol\beta \in E_\Phi}\mathbf{z}^{\boldsymbol\beta}$. When it is written in lowest terms, the denominator is $\Pi_{\boldsymbol\beta \in \textrm{CF}_\Phi}(1-\mathbf{z}^{\boldsymbol\beta}).$
\end{Thm}

Our description of the completely fundamental elements obtained earlier leads to the following result.

\begin{Cor} 
\begin{enumerate}
\item  Let $I=(x_1^{a_1}\cdots x_n^{a_n})$ and  $J=(x_1^{b_1}\cdots x_n^{b_n})$ be principal ideals in $R=k[x_1, \ldots, x_n]$. For all $i=1, \ldots, n$, let $p_i, q_i$ be relatively prime integers such that $a_i/b_i=p_i/q_i$, if $b_i \neq 0$. If $b_i=0$ let $p_i=1, q_i=0$.

The Hilbert series of $\B(I, J)$ as a function of $\mathbf{z}=(r, s, m_1, \ldots, m_n) $ has denominator equal to 

$$ \prod_{i=1}^{n}(1 -  m_i) \cdot \prod_{i=0}^{n+1} (1 - r^{q_i} s^{p_i}\prod_{j=1}^n  m_j^{\max(a_jq_i, b_j p_i)}),$$ when written in its lowest terms.
\item
Let $a,b \in \N$ and $\Phi=\Phi_{a,b}$. Let $p, q$ relatively prime such that $a/b=p/q$. The Hilbert series of $\B((x^a) \cap (x^b))$, as a function of $(r, s, m)$, has the denominator equal to 
\[(1-m)(1-rm^a)(1-sm^b)(1-r^qs^p\nu^{pb})\] when written in lowest terms.

\end{enumerate}
\end{Cor}
\begin{proof}
Consider the Hilbert series of $R_\Phi$ as a function of $3n+2$-variables, say $\mathbf{z}=(r, s, h_1, k_1, m_1, \ldots, h_n, k_n, m_n) $. We can use Theorems~\ref{S-HS} and~\ref{CF} to conclude that the Hilbert series of $R_\Phi$ has denominator equal to 
$$ \prod_{i=1}^{n}(1 -  h_i k_i m_i) \cdot \prod_{i=0}^{n+1} (1 - r^{q_i} s^{p_i}\prod_{j=1}^n  h_j^{\max(a_jq_i, b_j p_i) - a_jq_i} k_j ^{\max(a_jq_i, b_j p_i) - b_jp_i} m_j ^{\max(a_jq_i, b_j p_i)}),$$ when written in its lowest terms.

The isomorphism between $E_\Phi$ and $Q$ sends  $(r,s, h_1, k_1, m_1, \dots, h_n, k_n,m_n)$ to $ (r,s, m_1, \dots,m_n)$  which means that the Hilbert series of $\B$ is obtained from that of $R_\phi$ by replacing all $h_i, k_i$ by $1$, for $i=1, \ldots, n$. This produces the claim in (i) and a further specialization gives (ii).
\end{proof}

\medskip
\noindent
The canonical module of a Cohen-Macaulay ring $R$ is a very important classical concept that often can be hard to compute, but it is useful in the study of the homological properties of $R$. A canonical ideal for $R$ is simply an ideal of $R$ which is $R$-isomorphic to the canonical module of $R$. For standard facts about the canonical module, we refer the reader to Chapter 3 in~\cite{Bruns}.

\begin{Def} We say $\boldsymbol\beta \in E_\Phi$ is \emph{positive}, denoted $\boldsymbol\beta >0$, if each coordinate of $\boldsymbol\beta$ is positive. A positive $\boldsymbol\beta \in E_\Phi$ is called minimal if for any $\boldsymbol\gamma>0$, $\boldsymbol\gamma \in E_\Phi$, then $\boldsymbol\gamma-\boldsymbol\beta \geq 0$
\end{Def}

We are moving towards describing the canonical ideal of $\B((\mathbf{x}^\mathbf{a}),(\mathbf{x}^\mathbf{b}))$. The following result provides this description.

\begin{Thm}\textup{(Corollary 13.1 in \cite{Stanley})} Denote the canonical ideal of $R_\Phi$ by $\Omega(R_\Phi)$. Then
\[\Omega(R_\Phi) = k\{\mathbf{x}^{\boldsymbol\beta} : \boldsymbol\beta \in E_\Phi, \boldsymbol\beta >0\}.\]
\end{Thm}

Since any element in the canonical ideal is a linear combination of positive monomials, it is easy to see that the set of minimal positive elements provides a generating set for this ideal.

We will construct this set of minimal positive elements from $E_\Phi$ for our $\Phi$, and obtain a generating set for the canonical ideal $\Omega(\B)$.

\begin{Thm}\label{min}  Let $\mathbf{a,b} \in \N ^n$ and $\Phi=\Phi_{\mathbf{a,b}}$. Assume that $(\mathbf{a,b})$ is non-degenerate. Let $C$ be the set
\[\{\boldsymbol\gamma+(0,0,1, \ldots, 1)|\boldsymbol\gamma \in \textrm{FUND}_\Phi \setminus (\{ \boldsymbol\alpha_i, i=1, \ldots, n \} \cup \{\boldsymbol\beta(1,0), \boldsymbol\beta(0,1)\})\}.\]
Then the set of minimal positive $\boldsymbol\beta \in E_\Phi$ is equal to $C$.
\end{Thm}

\begin{proof}

For every $i=0, \ldots, n-1$, let $\cH'_i = \cH_i \setminus \cH_{i+1}$, and $\cH'_n=\cH_n$. These sets are pairwise disjoint.

First, we show that any minimal positive element of $E_\Phi$ is in $C$. Let $\boldsymbol\beta = (r, s, \ldots, h_t, k_t, m_t, \ldots) $ be a minimal positive element of $E_\Phi$.

Let us assume without loss of generality that $(r, s) \in \cH'_i$, for some $i =0, \ldots, n$. This means that $b_i s\leq a_i r$ and $ b_{i+1}s > a_{i+1}r$.

Then $a_ir+h_t=m_t=b_is+k_t$, and we can decompose $(r,s)$ into a sum of Hilbert basis elements $\sum_h l_h(r_h,s_h)$, with $H$ finite set of indices such that $l_h \geq 1$ for all $h \in H$.

We claim there that $\boldsymbol\beta \geq \boldsymbol\beta(r_{h}, s_{h})$, for every $h \in H$.

Let $1 \leq j \leq i$. 

Then
\[(r,s,h_j,k_j,m_j)=\left(\sum_hl_hr_h,\sum_hl_hs_h, m_j-a_j\sum_hl_hr_h,m_j-b_j\sum_hl_hs_h,m_j\right).\]

We need to show that 
\[(r,s,h_j,k_j,m_j)\geq (r_{h},s_{h},1,a_jr_{h}-b_js_{h}+1,ar_{h}+1).\]

This inequality is obviously satisfied for the first three components. Since $h_j,k_j \geq 1$, and $a_jr \geq b_js$ (as $j \leq i$), $m_j \geq a_jr+1$, so $m_j \geq a_jr_{h}+1$.

Lastly, since $k_j=m_j-b_js$ and $m_j \geq a_jr+1$, \[k_j \geq a_jr-b_js +1 =\sum_h l_{h}(a_jr_{h}-b_js_{h})+1,\] which gives what we wanted.

Now let $l \geq i+1$ and $l \leq n$.
Then
\[(r,s,h_l,k_l,m_l)=\left(\sum_hl_hr_h,\sum_hl_hs_h, m_l-a_l\sum_hl_hr_h,m_l-b_l\sum_hl_hs_h,m_l\right).\]

We need to show that 
\[(r,s,h_l,k_l,m_l)\geq (r_{h},s_{h},b_ls_{h}-a_lr_{h}+1,1, b_ls_{h}+1).\]

This inequality is satisfied for the first two and the fourth entry.

Now, $m_l = b_l s+ k_l \geq b_l s_{h}+1.$ Finally, $h_l = m_l - a_lr \geq b_ls+1-a_lr= \sum_h l_{h}(b_l s_{h}-a_lr_{h})+1 \geq b_l s_{h}-a_lr_{h}+1.$

Therefore, $\boldsymbol\beta \geq \boldsymbol\beta(r_{h}, s_{h})$ for all $h \in H$. But $\boldsymbol\beta$ is minimal and so there exists a unique $h$ such that $\boldsymbol\beta \geq \boldsymbol\beta(r_{h}, s_{h})$. But $\boldsymbol\beta$ is positive and a simple exercise shows that it must belong therefore to $C$.

Next, we show that any element from $C$ is minimal. First, note that any two elements of $A$ are incomparable: let $\boldsymbol\beta(r,s) \in C, (r,s) \in \cH_i$ and let
\[\boldsymbol\beta'=\boldsymbol\beta(r',s') \textrm{ for another } (r',s') \in \mathcal{H}_{i'}, \textrm{ with }\boldsymbol\beta(r',s')>\boldsymbol\beta(r,s).\]

We claim that $i'=i$. 

Assume that $i' < i$. Then because $i'+1 > i$, $\pi_{i'+1}(\boldsymbol\beta') = (r', s', b_{i'+1}s' - a_{i'+1} r' +1, 1, b_{i'+1} s'+1).$

On the other hand, because $i'+1 \leq i$, $\pi_{i'+1}(\boldsymbol\beta) =(r, s, 1, a_{i'+1}r-b_{i'+1}s +1 , a_{i'+1}r +1)$.

Since $\pi_{i'+1}(\boldsymbol\beta') > \pi_{i'+1}(\boldsymbol\beta)$ implies that $a_{i'+1}r-b_{i'+1}s =0$, so $(r,s) \in \cH_{i'}$.

But $(r,s) \in \cH_i$. Therefore $i=i'$. The case $i' > i$ is similar.

\medskip
\noindent

Note that, since $(r,s)$ and $(r',s')$ are both in $\cH'_i$, 
\[b_i s \geq a_i r, b_i s' \geq a_i r'\] and
\[b_{i+1} s < a_{i+1} r, b_{i+1}s' < a_{i+1} r'.\]

We claim $b_is'-a_ir'+1< b_is-a_ir+1$ or $a_{i+1}r'-b_{i+1}s' +1 < a_{i+1}r-b_{i+1}s +1$, since this is equivalent to
\[b_i(s'-s) < a_i(r'-r)\]
or 
\[a_{i+1}(r'-r) < b_{i+1}(s'-s)\]

Assume neither holds.  Then $(r'-r, s'-s) \in Q_i$.

But $$(r,s) + (r'-r, s'-s) = (r', s'), $$ and then the Hilbert basis element  $(r', s')$ is a sum of two elements in $Q_i$. Impossible. Therefore $\boldsymbol\beta$ and $\boldsymbol\beta'$ are incomparable.\\

Lastly, let $\boldsymbol\beta \in C$ not minimal in $E_\Phi$. Therefore there is an element $\boldsymbol\gamma< \boldsymbol\beta$. And, by the first half of this proof, $\boldsymbol\gamma$ is larger than some element $\boldsymbol\beta' \in C$.  So $\boldsymbol\beta > \boldsymbol\gamma > \boldsymbol\beta'$, which contradicts that all elements of $C$ are incomparable. Therefore any element of $C$ is minimal.
\end{proof}

\begin{Exa} \label{Ex4} Returning to Example \ref{Ex2} recall that 

\begin{align*} FUND_\Phi&=\{(0,0,1,1,1),(0,1,2,0,2), (1,3,1,0,6),(2,5,0,0,10), (1,0,0,5,5),\\
&\qquad(1,1,0,3,5),(1,2,0,1,5)\}.
\end{align*}

 We drop the first one, as that is $\boldsymbol\alpha_1$, as well as the second and fifth, as they are $\boldsymbol\beta(0,1)$ and $\boldsymbol\beta(1,0)$, respectively. Then the set of minimal elements is obtained by adding $(0,0,1,1,1)$ to each of the remaining elements, giving the set
\[\{(1,3,2,1,7),(2,5,1,1,11)(1,1,1,4,6),(1,2,1,2,6)\}.\]
\end{Exa}

\begin{Cor} Let $\mathbf{a,b} \in \N^n$, $\B=\B((\mathbf{x}^\mathbf{a}),(\mathbf{x}^\mathbf{b}))$, and $\Phi=\Phi_{\mathbf{a,b}}$. Then, using the notations in Theorem~\ref{min},  the canonical module of $\B$ is generated by the monomials $\mathbf{x}^{\boldsymbol\beta}$, where $\boldsymbol\beta \in C.$
\end{Cor}

\begin{Exa}
Our computations in Example~\ref{Ex4} show that the canonical ideal of $\B=\B((x^5), (x^2))$ is generated by $x^7uv^3, x^{11}u^2v^5, x^6uv, x^6uv^2.$
\end{Exa}

\begin{Cor}
\label{number}Let $\mathbf{a,b}$ and $\B$ be as above. Further assume that $(\mathbf{a,b})$ is non-degenerate. Then the number of minimal positive elements of $E_\Phi$, and thus the number of generators of the ideal $\Omega(\B)$ is $\sum_{i=0}^n|\mathcal{H}_i|-(n+2)$.
\end{Cor}
\begin{proof}

There is one fundamental element for each Hilbert basis element in $Q_i$, plus one for each of the variables $x_i$. However, the elements that live on the line of slope $a_i/b_i$ are double-counted, since they are in both $ \mathcal{H}_i \cap \mathcal{H}_{i+1}$, for all $i$. So $|\textrm{FUND}_\Phi|=\sum_{i=0}^n |\mathcal{H}_i|$. Then, by Theorem \ref{min}, we need to remove $n+2$ elements, so the statement of the Corollary follows.
\end{proof}

\begin{Exa} In Example \ref{Ex1}, $n=1$, $\cH_0=\{(0,1),(1,1),(2,1),(5,2)\}$ and $\cH_1=\{(1,0),(3,1),(5,2)\}$. So $|\cH_0|+|\cH_1|-(n+2)=4+3-3=4$. Indeed, we produced 4 minimal elements above.
\end{Exa}

Stanley's methods allow us to easily determine whether $\B$ is Gorenstein, since we only need to establish when the canonical ideal is principal.

\begin{Thm}\textup{(Corollary 13.2 in \cite{Stanley})} $R_\Phi$ is Gorenstein if and only if there exists a unique minimal $\boldsymbol\beta >0$ in $E_\Phi$ (i.e. if $\boldsymbol\gamma>0$, $\boldsymbol\gamma \in E_\Phi$, then $\boldsymbol\gamma-\boldsymbol\beta \geq 0$).
\end{Thm}

\begin{Cor} $\B$ is Gorenstein if and only if there exists a unique minimal $\boldsymbol\beta >0$ in $E_\Phi$.\end{Cor}

There are very few Gorenstein intersection algebras of principal monomial ideals. The next result establishes this.

\begin{Cor} Let $\B=\B(I,J)$ be the intersection algebra of two principal monomial ideals $I= (\mathbf{x}^\mathbf{a})$ and $J=(\mathbf{x}^\mathbf{b})$ in $R=k[x_1, \ldots, x_n]$ that is Gorenstein. Assume that $(\mathbf{a,b})$ is non-degenerate. Then $n=1$, $a=b$ and $I=J=(x^a)$.
\end{Cor}
\begin{proof} 

By Corollary~\ref{number}, if $\B$ is Gorenstein then $\sum_{i=0}^n|\mathcal{H}_i|-(n+2)=1$.

Since $(\mathbf{a,b})$ is non-degenerate, $|\mathcal{H}_i| \geq 2$ then this implies that $n+3 \geq 2(n+1)$ or $1\geq n$. So $n=1$.

Let $I=(x^a)$ and $J=(x^b)$, where $a,b \in \N$. First, note that that the non-degeneracy hypothesis implies that $a$ and $b$ are nonzero.

The number of minimal elements of $\B$ is equal to $|\mathcal{H}_0|+|\mathcal{H}_1|-3$, and since $\B$ is Gorenstein, $|\mathcal{H}_0|+|\mathcal{H}_1|=4$. Recall that both Hilbert bases must contain at least two elements, namely the generators of their cones. Let $q, p$ relatively prime such that $a/b = p/q$.

So $(0,1), (q,p) \in \mathcal{H}_0$ and $(1,0),(q,p) \in \mathcal{H}_1$, and since there are only four in total (including repetitions), these must be the only Hilbert basis elements.

The point $(1,1)$ must lie in one of the cones, say $Q_0$. If $(1,1)$ is a Hilbert basis element by assumption, there exists $n_1,n_2 \in \N$ such that 
\[(1,1)=n_1(0,1)+n_2(q,p)\]
which implies that $n_2q=1$ and $n_1+n_2p=1$. But $a,b,n_1,n_2 \in \N$, so $n_2=q=1$, and therefore $p=1$ as well.

In conclusion $a=b$. It can easily seen now that in this case the intersection algebra is Gorenstein.
\end{proof}

Examining these minimal positive elements leads us to an upper bound on the number of Hilbert basis elements for a fan formed by exactly two cones in the plane. This can be used to provide an upper bound for the minimal number of generators of the canonical ideal in the principal monomial case in one variable polynomial rings.

\begin{Thm} \label{bound} Let $R$, $a$, $b$ relatively prime, and $\Phi$ be as before. Assume without loss of generality that $a > b$. Set $a=bq+l$, with $1 \leq l \leq b-1$. Then the number of minimal positive elements of $E_\Phi$ is bounded above by $a-l+1$. In conclusion, the canonical ideal of $\B((x^a), (x^b))$ has at most $a-l+1$ generators.
\end{Thm}

\begin{proof} 

Let $(r,s,h,k,m)$ be an element of $E_\Phi$ and recall that $ar+h=m=bs+k$. First define $\mu=ar-bs=k-h$, and let $\mu>0$. So we may write our element as $(r,s,h,\mu+h,ar+h)$. 

Notice that the condition $\mu>0$ is equivalent to $ar-bs>0$, which is the same as considering all the points $(r,s) \in Q_1 \setminus Q_0$. Then $(1,1)$ is a minimal positive element in this cone: for this pair, $\mu=a-b>0$, and there is no other point with smaller positive coordinates. The smallest positive element in $E_\Phi$ corresponding to $(1,1)$ must have $h>0$, so the smallest such element is $(1,1,1,\mu+1,a+1)$. 

There are two sub-cases:\\ 
If $\mu>a-b$,
\begin{equation*}\begin{aligned}
(r,s,h,\mu+h,ar+h) &\geq(r,s,1,\mu+1,ar+1)\\
&\geq (1,1,1,1+\mu,a+1)\\
& \geq (1,1,1,1+a-b,a+1).\\
\end{aligned}
\end{equation*}

If $\mu\leq a-b$, we claim that for each such $\mu$, there exists a unique smallest $r_\mu, s_\mu$ such that $\mu=ar_\mu-bs_\mu$: since $ar-bs=\mu$ determines a line in the $(r,s)$ plane, let $(r_\mu,s_\mu)$ be the point with smallest integer coordinates on that line. Obviously, $(r,s)\geq (r_\mu,s_\mu)$, and the smallest positive element in $E_\Phi$ corresponding to $(r_\mu,s_\mu)$ is $(r_\mu,s_\mu,1,\mu+1,ar_\mu+1)$. Then for $\mu \leq a-b$,
\begin{equation*}\begin{aligned}
(r,s,h,\mu+h,ar+h) &\geq(r,s,1,\mu+1,ar+1)\\
&\geq (r_\mu,s_\mu,1,\mu+1,ar_\mu+1).\\
\end{aligned}
\end{equation*}

If $\mu=0$, then $ar=bs$, and, since $a$ and $b$ are relatively prime (because $l>0$), 
\begin{equation*}\begin{aligned}
(r,s,1,\mu+1,ar+1) &=(r,s,1,1,ar+1)\\
&\geq (b,a,1,1,ab+1).\\
\end{aligned}
\end{equation*}
So, for $\mu>0$, i.e. $Q_1 \setminus Q_0$, the number of minimal positive elements of $E_\Phi$ is no more than $a-b+1$: $a-b-1$ elements for each $0<\mu \leq a-b$, one for $\mu>a-b$, and one for $\mu=0$.\\

Now consider the points of $E_\Phi$ that correspond to pairs $(r,s) \in Q_0$, i.e. where $bs-ar=\eta>0$. Then $\eta=h-k$, so we can write a full element of $E_\Phi$ as $(r,s,\eta+k,k, bs+k)$.

Since $a=bq+l$, we claim that the smallest $(r,s)$ where $bs-ar>0$ (i.e. the smallest pair $(r,s) \in Q_1$), is $(1,q+1)$. This pair is certainly in $Q_0$, since $\eta=b(q+1)-a=bq+b-a=b-l>0$. Also, $s>q$, because $\eta >0$ and so
\[s=\frac{ar+\eta}{b} >q = \frac{(bq+l)r+\eta}{b} = \frac{bqr+lr+\eta}{b}>q.\]
Since $s \in \N$, $s\geq q+1$.
So, if $\eta \geq b-l$,
\begin{equation*}\begin{aligned}
(r,s,\eta+k,k,bs+k) &\geq (r,s,1+\eta,1,bs+1)\\
&\geq (1,q+1,1+\eta,1,b(q+1)+1)\\
&\geq(1,q+1,1+b-l,1,b(q+1)+1).
\end{aligned}
\end{equation*}

If $0 < \eta < b-l$, we will minimize for every $\eta$ as we did for every $\mu$ before. The same logic applies, and for every $(r,s)$ with $\eta <b-l$, there is a unique smallest $(r_\eta, s_\eta)\leq (r,s)$. Then  
\begin{equation*}\begin{aligned}
(r,s,h+\eta,h,bs+h) &\geq(r,s,1+\eta,1,bs+1)\\
&\geq (r_\eta,s_\eta,1,1+\eta,bs_\eta+1).\\
\end{aligned}
\end{equation*}

So for $\eta>0$, we have at most $1+b-l-1=b-l$ minimal positive elements, giving a total of $a-b+1+b-l=a-l+1$ at most minimal positive elements for $E_\Phi$.
\end{proof}

\begin{Exa}
Recall that in Example \ref{Ex1}, $a=5$ and $b=2$, so $a>b$. Then $5=2\cdot q+l$ gives $q=2$ and $l=1$, so the number of minimal positive elements of $E_\Phi$ is bounded above by $5-1+1=5$. This bound is not sharp, as we saw in \ref{Ex3} that the number of minimal positive elements is actually 4.
\end{Exa}

\section*{Acknowledgements}
The first author thanks Mel Hochster for suggestions to use Stanley's work on linear diophantine equations in the study of the intersection algebra for principal monomial ideals.  Many thanks to the anonymous referee and Sandra Spiroff for comments that improved the exposition. Some of the work in this paper was part of the doctoral dissertation of the second author completed at Georgia State University under direction of the first author.

\bibliography{References}

\end{document}